\documentclass{elsarticle}
\usepackage{{amsmath, amsthm, amssymb, amsfonts}}
\usepackage{color}
\usepackage{verbatim}
\usepackage{lscape} %landscape modus, needed for large tabular/array
\usepackage{caption}

\newtheorem{thm}{Theorem}
\newtheorem{lem}[thm]{Lemma} 
\newtheorem{prop}[thm]{Proposition} 
 
\newtheorem{cor}[thm]{Corollary} 

\newdefinition{defn}{Definition}
\newdefinition{exmp}{Example} 
\newdefinition{remark}{Remark}

\DeclareMathOperator{\ord}{ord}

\newcommand{\pglq}{\mathrm{PGL}_2(\mathbb{Q})}
\newcommand{\pgloq}{\mathrm{PGL}_2(\overline{\mathbb{Q}})}
\newcommand{\Bmax}{B_{\mathrm{max}}}

\begin{document}
\title{A group-invariant version of Lehmer's conjecture on heights}
\author{Jan-Willem M. van Ittersum}
\ead{j.w.m.vanittersum@uu.nl}
\address{Mathematisch Instituut, Universiteit Utrecht, Postbus 80.010, 3508 TA Utrecht, The~Netherlands}

\date{\today}

\begin{abstract}
We state and prove a group-invariant version of Lehmer's conjecture on heights, generalizing papers by Zagier (1993) \cite{zag93} and Dresden (1998) \cite{dre98} which are special cases of this theorem. We also extend their three cases to a full classification of all finite cyclic groups satisfying the condition that the set of all orbits for which every non-zero element lies on the unit circle is finite and non-empty.
\end{abstract}

\maketitle

\section{A Lehmer-type problem for the Weil height}
The \textit{Mahler measure} of a non-zero polynomial $f\in\mathbb{Z}[x]$ given by
\begin{equation}\label{eq:f} f(x)=a_n \prod_{i=1}^n (x-\alpha_i) \end{equation}
is defined as
$$M(f) = |a_n| \prod_{i=1}^n \max(|\alpha_i|,1).$$
In 1933, Lehmer asked whether there exists a lower bound $D>1$ such that for all $f\in\mathbb{Z}[x]$ it holds that
$$M(f) = 1 \quad \text{or} \quad M(f)\geq D.$$
He showed that if such a $D$ exists, then $D \leq 1.1762808\ldots$, the largest real root of the polynomial $x^{10}+x^9-x^7-x^6-x^5-x^4-x^3+x+1$ \cite{leh33}. Nowadays, this is still the smallest known value of $M(f)>1$ for $f\in \mathbb{Z}[x]$. 

Mahler's measure is related to the Weil height of an algebraic number. Let $K$ be an algebraic number field and $v$ a place of $K$. We assume that this $v$-adic valuation is normalized in such a way that for all non-zero $\alpha \in K$ the product of $|\alpha|_v$ over all places $v$ is equal to $1$ and the product of $|\alpha|_v$ over all Archimedean $v$ is equal to the absolute value of $N_{K/\mathbb{Q}}(\alpha)$. Then, for $\alpha \in K^*$ the (logarithmic) Weil height $h$ is defined by 
$$h(\alpha) = \frac{1}{[K:\mathbb{Q}]}\sum_{v} \log^+|\alpha|_v,$$
where the sum is over all places $v$ of $K$. We used the notation $\log^+(z)$ to refer to $\log\max(z,1)$ for $z\in \mathbb{R}$. The Weil height is independent of $K$ and if the polynomial (\ref{eq:f}) is the minimal polynomial of $\alpha$ over $\mathbb{Q}$, then $h(\alpha)=\tfrac{1}{n} \log M(f)$. This Weil height can be extended to $\mathbb{P}^1(\overline{\mathbb{Q}})$. %Let $d=[K:\mathbb{Q}]$ and $d_v=[K_v:\mahtbb{Q}_v]$. 
Namely, for $x=[x_1:x_2] \in \mathbb{P}^1(\overline{\mathbb{Q}})$, we define
$$h(x) = \frac{1}{[K:\mathbb{Q}]}\sum_v \log\max(|x_1|_v,|x_2|_v),$$
where  $K$ is chosen such that $x_1, x_2\in K$. Note that $h(x)\geq 0$.

\begin{defn} Let $G$ be a finite subgroup of $\pglq$. The \emph{$G$-orbit height of $x \in \mathbb{P}^1(\overline{\mathbb{Q}})$} is defined as
$$h_G(x) = \sum_{\sigma \in G} h(\sigma x).$$
\end{defn}
Note that $h_G(x)\geq 0$ and $h_G(\sigma \alpha)=h_G(\alpha)$ for all $\sigma \in G$. We can now state the \emph{$G$-invariant Lehmer problem}, namely: given a finite group $G$ does there exist a positive lower bound $D$ such that
\begin{equation}\label{eq:problem} h_G(x) = 0 \quad \text{or} \quad h_G(x)\geq D \end{equation}
for all $x \in \mathbb{P}^1(\mathbb{\overline{Q}})$? As Zagier pointed out \cite{zag93}, if $G$ is trivial such a constant does not exist (e.g., $h_{\{e\}}(\sqrt[n]{2})=n^{-1}\log2 \to 0$). Assuming a mild restriction on $G$, which we will state next, we will prove that this lower bound $D$ exists for $h_G$. 

Recall that as a consequence of Kronecker's lemma \cite{kro}, for $\alpha\in K$ we have that $h(\alpha)=0$ if and only if $\alpha=0$ or $\alpha$ is a root of unity. We will now define a set $\mathcal{O}$ of orbits such that elements of these orbits are precisely the zeros of $h_G$ over $K$.

\begin{defn} Let $\mathcal{Q}$ be the set of all orbits of the action of $G$ on $\hat{\mathbb{C}}=\mathbb{C}\cup \{\infty\}$.
Let $\mathcal{O} \subset \mathcal{Q}$ be the set of all orbits for which every non-zero element lies on the unit circle, i.e.
$$\mathcal{O}=\{O\in\mathcal{Q} \mid \forall z\in O: z=0 \text{ or } |z|=1\}.$$
\end{defn}

The main purpose of this note is to solve the $G$-invariant Lehmer problem in the case that $h_G$ has finitely many zeros:

\begin{thm}\label{thm:1}
If $\mathcal{O}$ is finite, then there exists a positive $D$ such that
$$h_G(\alpha)=0 \quad \text{or} \quad h_G(\alpha)\geq D$$ 
for all $\alpha \in K$.
\end{thm}

\begin{remark}
For a finite $G\leq \pglq$ this theorem can also be stated in terms of Mahler measures instead of heights. Let $\alpha$ be a given algebraic integer with minimal polynomial $f \in \mathbb{Z}[x]$ of degree $n$. Assume $\sigma \in G$ and write $\sigma =\left(\begin{smallmatrix}a & b\\ c& d\end{smallmatrix}\right)$. Let
$$f_\sigma(z) = C (cz+d)^n f(\sigma(z)),$$
where $C\in\mathbb{Q}$ is chosen in such a way that $f_\sigma$ is primitive, so that $f_\sigma$ is the minimal polynomial of $\sigma^{-1}\alpha$. Then, this theorem implies that there exists a constant $E>1$ such that
$$\prod_{\sigma \in G}M\left(f_\sigma\right) = 1 \quad \text{or} \quad \prod_{\sigma \in G}M\left(f_\sigma\right)\geq E^n$$
for all primitive irreducible polynomials $f\in \mathbb{Z}[x]$.
\end{remark}

Later we will see that if $\mathcal{O}$ is infinite, its subsets contain all roots of unity. Moreover, if $\mathcal{O}$ is finite, its subsets can only contain $0$ and the roots of cyclotomic polynomials of degree at most 2. For the cyclic case $G=\langle \sigma \rangle$, we will use this to classify all $\sigma \in \pglq$ for which $\mathcal{O}$ is finite and non-empty. In nearly all cases, it is also possible to calculate the maximal value of $D$ for which Theorem~\ref{thm:1} holds. 
For three of these cases, these values are already known. By a theorem of Zhang, for which Zagier gave an elementary proof, we have that $D = \tfrac{1}{2}\log\frac{1+\sqrt{5}}{2} = 0.2406059\ldots$ for $G = \{z,1-z\}$ \cite{zha92, zag93}. Here we identified $\pglq$ with the M\"obius transformations, maps $\sigma:\hat{\mathbb{C}}\to \hat{\mathbb{C}}$ of the form $\sigma(z)=\frac{az+b}{cz+d}$ for $a,b,c,d\in\mathbb{Q}$. Dresden proved that $D = \log|\beta|= 0.4217993\ldots$ for $G = \{z,\frac{1}{1-z},1-\frac{1}{z}\}$, where $\beta$ is a maximal root in absolute value of $\left(z^2-z+1\right)^3-\left(z^2-z\right)^2$ and mentioned that for $G = \{z,\frac{z+1}{-z+1},-\frac{1}{z},\frac{z-1}{z+1}\}$ one has $D = \log|\gamma| = 0.7328576\ldots$ for $\gamma$ a maximal root in absolute value of $\left(z^2+1\right)^4+z^2\left(z^2-1\right)^2 $ \cite{dre98}.

\section{Proof of Theorem~\ref{thm:1}}
Set $\mathcal{O}=\{O_i\mid i\in\{1,2,\ldots, k\}\}$. For each orbit $O_i \in \mathcal{O}$ we choose $\alpha_i\in O_i$ and define $p_i\in\mathbb{Z}[x]$ as the minimal polynomial of $\alpha_i$. Let $N$ be the maximum of the degrees of all the $p_i$ and let
$$n_v= \begin{cases}
			0 & \text{if } v \text{ is non-Archimedean,}\\
			1 & \text{if } v \text{ is real,}\\
			2 & \text{if } v \text{ is complex.}
		\end{cases}$$
The proof of Theorem~\ref{thm:1} will follow directly from the following lemma:	
\begin{lem}\label{lem:1}
Let $v$ be a place of $K$ and $\alpha \in K$. There exists $\Bmax>0$ such that for all $B$ with $0<B<\Bmax$ there exists a positive $D$ such that
\begin{align}\label{eq:lem1}\sum_{\sigma \in G}\left(\log^+|\sigma(\alpha)|_v-\tfrac{1}{2}\log|\sigma(\alpha)|_v-B\sum_{i=1}^k \log|p_i(\sigma(\alpha))|_v \right)\geq  n_v D.
\end{align}
\end{lem}

\begin{proof}
Firstly, assume $v$ is finite. We will show that the summand 
\begin{align}\label{eq:summand}
\log^+|\sigma(\alpha)|_v-\tfrac{1}{2}\log|\sigma(\alpha)|_v-B\sum_{i=1}^k \log|p_i(\sigma(\alpha))|_v
\end{align} of (\ref{eq:lem1}) is nonnegative for all $\sigma \in G$. If $\sigma(\alpha)$ is integral at $v$, it follows that $\log^+|\sigma(\alpha)|_v=0$ and $\tfrac{1}{2}\log|\sigma(\alpha)|_v \leq 0.$
By writing $p_j(\sigma(\alpha)) = b_{j n} \sigma(\alpha)^n + \ldots + b_{j 0}$ for $j\in\{1,2,\ldots, k\}$ and $b_{j i}\in\mathbb{Z}$ we obtain
$$|p_j(\sigma(\alpha))|_v \leq \max(|b_{jn}|_v\cdot {|\sigma(\alpha)|_v}^n, \ldots ,|b_{j0}|_v)\leq 1.$$
Therefore, $\sum_{i=1}^k \log|p_i(\sigma(\alpha))|_v \leq 0,$ which implies that the summand (\ref{eq:summand}) is nonnegative for all $B \in \mathbb{R}^+$.\\
If $|\sigma(\alpha)|_v > 1$, we find that $\log^+|\sigma(\alpha)|_v-\tfrac{1}{2}\log|\sigma(\alpha)|_v =\tfrac{1}{2}\log|\sigma(\alpha)|_v >0.$
 Using the same notation as above,
$$\log|p_j(\sigma(\alpha))|_v \leq \log\max(|b_n|_v\cdot {|\sigma(\alpha)|_v}^n, \ldots ,|b_0|_v)\leq N \log|\sigma(\alpha)|_v.$$
Therefore, for 
$$\Bmax\leq \frac{1}{2 kN} \leq \frac{\log|\sigma(\alpha)|_v}{2 \sum_{i=1}^k \log|p_i(\sigma(\alpha))|_v}$$
the summand (\ref{eq:summand}) is positive for all $B$ with $0<B<\Bmax$ and all $\sigma \in G$. 

Secondly, if $v$ is Archimedean, then $|\alpha|_v = |\iota(\alpha)|^{n_v}$ for some embedding $\iota$ of $K$ into $\mathbb{C}$. Let
$$g_1(z) = \sum_{\sigma \in G} \left(\log^+|\sigma(z)|-\tfrac{1}{2}\log|\sigma(z)|\right) \quad \text{and} \quad g_2(z) =- \sum_{\sigma \in G}\sum_{i=1}^k \log|p_i(\sigma(z))|.$$
The claim is that for $z\in\mathbb{C}$ the function
$$ f(z)=g_1(z)+B\cdot g_2(z)= \sum_{\sigma \in G} \left(\log^+|\sigma(z)|-\tfrac{1}{2}\log|\sigma(z)|-B\sum_{i=1}^k \log|p_i(\sigma(z))| \right)$$
is bounded below by some constant $D>0$. Clearly, $f$ tends to infinity as $\sigma(z)$ tends to zero or to one of the roots of the $p_i$. As $\sum_{i=1}^k \log|p_i(\sigma(z))| \leq k\log|C\sigma(z)^N|$ for some $C\in \mathbb{R}^+$ and $\sigma(z)$ sufficiently large, assuming $\Bmax < \frac{1}{2kN}$ we find that if $\sigma(z)$ tends to infinity then $f$ tends to infinity. As $f$ is continuous elsewhere and harmonic if $|\sigma(z)|\neq 1$ for all $\sigma \in G$, it attains a minimum on a circle $|\sigma(z)|=1$ for some $\sigma \in G$. As $f(z)=f(\sigma(z))$ for all $\sigma\in G$, we can assume that this minimum is attained on the unit circle. If this minimum is strictly positive, we are done. Otherwise, $g_2(z)\leq 0$ and this minimum is attained in the set $S=\{z\in \mathbb{C} \mid |z|=1 \text{ and } g_2(z)\leq 0\}$. Let $q\in S$ and let $Q$ be the orbit of $q$.  If $Q \in\mathcal{O}$, write $Q=O_i$. Then, there is a $\tau\in G$ such that $\tau(q)$ is a root of $p_i$. It follows that $g_2$ tends to infinity as $z$ tends to $q$, contradicting $q\in S$. Therefore, $Q\not\in \mathcal{O}$. Hence, there exists a $\tau \in G$ such that $\tau(q) \neq 0$ and $|\tau(q)|\neq 1$. This implies that $\log^+|\tau(q)|-\tfrac{1}{2}\log|\tau(q)|>0$. As for all $z\in \mathbb{C}$ we have that $\log^+|z|-\tfrac{1}{2}\log|z|\geq 0$, it follows that $g_1(q)>0$ for all $q\in S$. As $S$ is compact, $g_1$ attains a minimum $m>0$ in $S$. Also, $g_2$ attains a minimum $n$ in $S$. Letting $\Bmax<-m/n$, it follows that $f$ attains a positive minimum $D$ in $S$. 
\end{proof}

\begin{proof}[Proof of Theorem~\ref{thm:1}]
Observe that for $\beta\in K^*$ we have
\begin{align}\label{eq:height} \sum_v n_v = [K:\mathbb{Q}] \quad \text{and} \quad \sum_v \log|\beta|_v = 0.\end{align}
Then, for $\alpha$ for which there is no $\sigma \in G$ such that $\sigma(\alpha)$ is zero, infinite or a root of some $p_i$, we can sum the inequality (\ref{eq:lem1}) in Lemma~\ref{lem:1} over all places $v$ of $K$ and apply (\ref{eq:height}). After dividing by $[K:\mathbb{Q}]$ we find that $ h_G(\alpha)\geq D$ for all but finitely many $\alpha\in K$. Hence, it follows that for some possibly smaller value of $D$ and all $\alpha \in K$ we have $h_G(\alpha)=0$ or $h_G(\alpha)\geq D$. 
\end{proof}

\section{When is $\mathcal{O}$ finite?}
We will investigate how strong the condition is that $\mathcal{O}$ is finite. 

\begin{prop} The set $\mathcal{O}$ is finite if and only if there exists a root of unity $\zeta$ such that $h_G(\zeta)> 0$.  \end{prop}
\begin{proof} If $\mathcal{O}$ is finite, we can choose a root of unity $\zeta$ and $\tau \in G$ with $|\tau \zeta| \neq 0,1$. As $h(\sigma \zeta)\geq 0$ for all $\sigma \in G$ and $h(\tau\zeta)>0$, we find $h_G(\zeta)> 0$. 

Conversely, if there exists a root of unity $\zeta$ with $h_G(\zeta) > 0$, then there exists a $\tau \in G$ such that $\tau \zeta$ is not a root of unity. It is known that M\"obius transformations map real circles on $\hat{\mathbb{C}}$ to real circles on $\hat{\mathbb{C}}$ provided that we regard a line through $\infty$ as a circle. Hence, $\tau$ maps the unit circle to another circle. As two circles intersect in at most two different points, there are at most two roots of unity $\eta$ such that $\tau \eta$ is also a root of unity. Hence, $\mathcal{O}$ is finite. 
\end{proof}

\begin{cor}
Let $G\leq \pglq$ be finite. Then, $\mathcal{O}$ is infinite if and only if $G$ is a subgroup of
\begin{align}\label{eq:fingrps}\left\{I,
\begin{pmatrix}
0 & 1 \\
1 & 0
\end{pmatrix}, 
\begin{pmatrix}
a & b\\
-b& -a
\end{pmatrix} ,
\begin{pmatrix}
b & a\\
-a & -b
\end{pmatrix}\right\}\end{align} for some $a,b\in\mathbb{Q}$ with $a^2\neq b^2$. \end{cor}
\begin{proof} If $\mathcal{O}$ is infinite then $\sigma(1)=\pm 1$ and $\sigma(-1)=\mp 1$. Hence, an element $\sigma\in G$ is of the form
$$\begin{pmatrix} a & b\\ b & a\end{pmatrix} \quad \text{or} \quad \begin{pmatrix} a & b \\ -b & -a\end{pmatrix}$$
for some $a,b\in\mathbb{Q}$ with $a^2\neq b^2$. The first is of infinite order unless $a=0$ or $b=0$. The product of two elements
$$\begin{pmatrix}
a & b\\
-b& -a
\end{pmatrix} \quad \text{and} \quad 
\begin{pmatrix}
a' & b'\\
-b'& -a'
\end{pmatrix}$$
of $\pglq$ is of finite order if and only if $ab'=ba'$ or $aa'=bb'$. Hence, $G$ must be a subgroup of (\ref{eq:fingrps}). It can easily be checked that $\mathcal{O}$ is infinite for subgroups of (\ref{eq:fingrps}).
 \end{proof}

\section{When is the $G$-orbit height zero?}
Denote with $\pm\tfrac{1}{2}\pm'\tfrac{1}{2}i\sqrt{3}$ the four primitive third and sixth roots of unity, where the sign $\pm'$ can be chosen independently from the sign $\pm$. 
\begin{lem} Let $\mathcal{O}$ be finite and $G\leq \pglq$. Then $h_G(\alpha)=0$ if and only if $\alpha$ equals $0, \pm 1, \pm i$ or  $\pm\tfrac{1}{2}\pm'\tfrac{1}{2}i\sqrt{3}$. %Moreover, there are at most two non-rational solutions $\alpha$ of $h_G(\alpha)=0$. 
\end{lem}
\begin{proof}  If $h_G(\alpha)=0$, we have for all $\sigma \in G$ that $\sigma\alpha$ equals $0$ or is a root of unity. Assuming $\alpha \in \mathbb{Q}$, we find $\alpha$ equals $0$ or $\pm 1$. If $\alpha \not \in \mathbb{Q}$, then for all $\sigma\in G$ we also have $\sigma \alpha \not \in \mathbb{Q}$, so $\sigma \alpha$ is a root of unity. 
By the proof of the previous proposition there is a $\tau\in G$ which maps at most two roots of unity to other roots of unity. %So, if $\alpha$ has more than two algebraic conjugates, there is a conjugate $\alpha'$ of $\alpha$ with $h_G(\alpha')\neq 0$. 
As $h_G(\alpha)=h_G(\alpha')$ for all algebraic conjugates $\alpha'$ of $\alpha$, it follows that the minimal polynomial of $\alpha$ has degree at most two. Hence, $\alpha$ equals $\pm i$ or $\pm\tfrac{1}{2}\pm'\tfrac{1}{2}i\sqrt{3}$. %Moreover, for all other non-rational choices of $\alpha$ we find $h_G(\alpha)\neq 0$. 
\end{proof}

\section{Explicit constants}\label{sec:5}

Dresden proved that all finite subgroups of $\pglq$ are isomorphic to the cyclic group $C_n$ or the dihedral group $D_n$ (where the latter is of order $2n$) for $n=1,2,3,4$ or $6$ \cite{dre04}. By the previous lemma, there are only $9$ possible elements in orbits in $\mathcal{O}$. Hereby, it is possible to determine all finite cyclic groups $G\leq \pglq$ for which $\mathcal{O}$ is finite and non-empty. Moreover, by generalizing Zagier's and Dresden's proofs \cite{zag93,dre98}, it is possible to explicitly find the best value of $D$ in Theorem~\ref{thm:1}. We have collected these data in Table~\ref{tab:1}, meaning the following: the first column of this table gives a list of one generator $\sigma \in \pglq$ for every cyclic group $G=\langle \sigma \rangle$ satisfying the condition that $\mathcal{O}$ is finite and non-empty. It is assumed that $p,q\in\mathbb{Z}$ are relatively prime such that $q>0$, $\det(\sigma)\neq 0$, $p/q\neq 0$ in the first and fifth row and $p/q \neq \mp 1/2$ in the seventh row. The second column shows the elements of $\mathcal{O}$, where we use the shorthand notation $\omega_{\pm \pm'}=\pm\tfrac{1}{2}\pm'\tfrac{1}{2}i\sqrt{3}$. The third column shows the order of $\sigma$. The remaining columns give information needed to determine the optimal value of $D$. We let $\phi(z) = \frac{1}{E} \prod_{\sigma \in G}p_1(\sigma(z))$ where $p_1$ corresponds to an orbit in $\mathcal{O}$ as in the proof of Theorem~\ref{thm:1} and $E$ is such that the numerator and denominator of $\phi(z)$ are relatively prime. The inequality 
$\displaystyle{\sum_{i=0}^{\ord\sigma-1}} \log^+|\sigma^i(z)|-B \log|\phi(z)|\geq D$
holds where the values of $B$ and $\exp(D)$ can be found in the corresponding row of the columns `$B$' and `$\exp(D)$'. %This inequality is comparable to the Archimedian case in \Cref{lem:1}. %When $B$ is displayed with a floating point, its exact value (in terms of $\alpha$) can be calculated.
From this, in a similar fashion as the proof of Theorem~\ref{thm:1}, one deduces
$ h_G(\alpha)=0$ or $h_G(\alpha)\geq D$ for all $\alpha \in \overline{\mathbb{Q}}$. The element $\alpha$ in the last column is such that equality holds in $h_G(\alpha)\geq D$. However, it is not unique. Here, a maximal root of a polynomial is defined as a root which

%\newgeometry{top=2.8 cm, bottom=2.8cm, left=2cm, right=2cm}
\begin{landscape}
\captionof{table}{Classification of cyclic groups with $\mathcal{O}$ non-empty and finite, together with data to find the optimal constant $D$ of Theorem~\ref{thm:1}. In Section~\ref{sec:5} the meaning of these data is explained.\vspace{-20pt}} \label{tab:1}
$$
\begin{array}{l l l l l l l l} \hline \vspace{-10pt}\\
&\sigma & \text{Elts. of } \mathcal{O} & \ord \sigma & B & \exp(D) & \begin{array}{@{}l@{}} D\approx \\ \scriptstyle{(\text{for } p/q = 5)} \end{array} & \text{Equality}\\ \hline 

&\begin{pmatrix} 1 & 0  \\ p/q & -1 \end{pmatrix}
	& \{0 \}
	& 2 
	& 1 
	& \max(||p|-q|,|q|) 
	& 1.38629 
	&  \begin{cases} 
   		1        & \text{if } p/q>0 \\
   		-1       & \text{if } p/q<0
  		\end{cases}		
  \\ \vspace{5pt}
  
&\begin{pmatrix} 1 & \pm 1  \\ p/q & -1 \end{pmatrix} 
	& \{0 , \mp 1 \}
	& 2 
	& 1 
	& \max(|p\mp q|/2, |q|) 
	& 0.69315 
	& \pm 1 \quad \text{ if } 2\mid p\mp q
	\\ \vspace{5pt}
		
&\begin{pmatrix} 1 & p/q  \\ p/q\pm 2 & -1 \end{pmatrix} 
	& \{\pm 1\} 
	& 2 
	& 1 
	& \max(|p\pm 3q|/4,|p\mp q|/4) 
	& 0.69315 		
	&\mp 1 \quad \text{ if } 	4\mid p\mp q
	\\ \vspace{5pt}
		
% \begin{pmatrix} 1 & p/q  \\-p/q & -1 \end{pmatrix}  & \{1 , -1\} & 2\\ \vspace{5pt}

&\begin{pmatrix} 1 & -1  \\ 3 & 1 \end{pmatrix} 
%	,\begin{pmatrix} 1 & 1  \\ -3 & 1 \end{pmatrix}  
	& \{0 , -1 , 1 \}
	& 3 
	& 1 
	& 5 
	& 1.60944 
	& i
	\\ \hline \vspace{5pt}
	
&\begin{pmatrix} 1 & p/q  \\ p/q & -1 \end{pmatrix}  
	& \{i , -i\} & 2 
	& \tfrac12
	& ||p|+q|/2 
	& 1.09861
	& 1,-1 \quad \text{if } 2\mid p+q
	\\ \vspace{5pt}
		
\text{\cite{dre98}}&\begin{pmatrix} 1 & 1  \\ -1 & 1 \end{pmatrix} 
%    ,\begin{pmatrix} 0 & 1  \\ -1 & 0 \end{pmatrix} 
%    \begin{pmatrix} 1 & -1  \\ 1 & 1 \end{pmatrix}
     & \{i\}, \{-i \} & 4 
     & 0.19408\ldots %\frac{\alpha ^4-1}{\alpha ^4-6 \alpha ^2+1}
     & |\alpha| 
     & 0.73286 
     & \begin{array}{@{}c@{}} \alpha, \text{maximal root of} \\
  			(z^2+1)^4+z^2(z^2-1)^2 \end{array}  
  	\\ \hline \vspace{5pt}
     
&\begin{pmatrix} 1 & p/q  \\ p/q\pm 1 & -1 \end{pmatrix} 
	& \{ \omega_{\pm +} , \omega_{\pm -} \}
	& 2 
	& \tfrac12
	&\max(|p\pm 2q|/3,|p\mp q|/3) 
	& 0.84730 
	& \mp 1 \quad \text{if } 3\mid p\mp q 
	\\ \vspace{5pt}  
	 
\text{\cite{dre98}}&\begin{pmatrix} 0 & 1  \\ -1 & \pm 1 \end{pmatrix} 
	& \{ \omega_{\pm +}\},\{\omega_{\pm -} \}
	& 3 
	& 0.11724\ldots % \frac{\alpha ^2\mp\alpha +1}{2 \left(2 \alpha ^2\pm\alpha -1\right)} 
 	& |\alpha| & 0.42180 & \begin{array}{@{}c@{}} \alpha,     
 		\text{maximal root of} \\   (z^2\mp z+1)^3-(z^2\mp z)^2 \end{array} 
 	\\ \vspace{5pt}
 
&\begin{pmatrix} 2 & \mp1  \\ \pm1 & 1 \end{pmatrix} 
    & \{\omega_{\pm +} \},  \{\omega_{\pm -} \}
    & 6 
    & 0.30503\ldots % \frac{2 \alpha ^6-6 \alpha ^5+10 \alpha ^3-15 \alpha ^2+9 \alpha -1}{2 \alpha ^6-6 \alpha ^5-15 \alpha ^4+40 \alpha ^3-15 \alpha ^2-6 \alpha +2}
    & \left|\frac{\alpha  (2 \alpha -1)}{\alpha +1}\right| & 1.75737 
    &   \begin{array}{@{}c@{}} \alpha, \text{maximal root of} \\
  		(z^2\mp z+1)^6+\\
  		z^2(2z^2\mp 5z+2)^2(z^2-1)^2
  		\end{array}
  \\ \hline   
\end{array}
$$

\end{landscape}
\captionof{table}{Cyclic groups $\langle \sigma \rangle$ with $\mathcal{O}$ non-empty and finite which appear twice in Table~\ref{tab:1}, together with data to find the optimal constant $D$ of Theorem~\ref{thm:1}. In Section~\ref{sec:5} the meaning of these data is explained.\vspace{-20pt}}\label{tab:2} 
$$
\begin{array}{l l l l l l l l l l l} \hline \vspace{-10pt}\\
&\sigma & \text{Elts. of } \mathcal{O} & B_1 & B_2 & \exp(D) & D\approx & \text{Equality}\\ \hline 

&\begin{pmatrix} 1 & 0  \\ \pm 1 & -1 \end{pmatrix}
     &  \{ 0 \}, \{\omega_{\pm +} , \omega_{\pm -}\}  
     & 1/2 & 1/2     
     &  \sqrt{\frac{1+\sqrt{5}}{2}}
     & 0.24061
     & \mp e^{\tfrac{2\pi i}{5}}\\  \vspace{5pt}
     
&\begin{pmatrix} 1 & 0  \\ \pm 2 & -1 \end{pmatrix} 
     &  \{0 \}, \{ \pm 1\}  
     & 2/3 & 1/3 
     &  \sqrt{3}
     & 0.54931
     & \omega_{\pm +}\\ \vspace{5pt}
     
\text{\cite{zag93, zha92}}&\begin{pmatrix} 1 & \mp 1  \\ 0 & -1 \end{pmatrix} 
     &  \{0 , \pm 1\}, \{\omega_{\pm +} , \omega_{\pm -}\}  
     & \frac{\sqrt{5}-1}{2\sqrt{5}}
     & \frac{1}{4\sqrt{5}}  
     & \sqrt{\frac{1+\sqrt{5}}{2}} 
     & 0.24061
     & \pm e^{\tfrac{\pi i}{5}}\\ \vspace{5pt}    
     
&\begin{pmatrix} 1 & \mp 1  \\ \mp 1 & -1 \end{pmatrix} 
     &  \{0 , \pm 1\}, \{i , -i\} 
     &  &   
     & %\sqrt{\frac{2+\sqrt{2}}{2}} 
     & %\pm e^{\tfrac{\pi i}{4}} 
     \\ \vspace{5pt} 
    
&\begin{pmatrix} 1 & \mp 1  \\ \mp 2 & -1 \end{pmatrix} 
     &  \{0 , \pm 1\}, \{\omega_{\mp +} , \omega_{\mp -}\}   &  &   &  
     &  % e^{\tfrac{2\pi i}{9}} (?) 
     \\ \vspace{5pt} 

&\begin{pmatrix} 1 & \mp 1  \\ \mp 3 & -1 \end{pmatrix} 
     &  \{0 , \pm 1\}, \{ \mp 1 \} 
     &  %&      \frac{8 \log 5}{15 \log 5 - 8 \log 2} 
     &  %\frac{7 \log 5}{15 \log 5 - 8 \log 2}  
     & %\sqrt{5} 
     &  %i
     \\ \hline 
\end{array}
$$
%\captionsetup{singlelinecheck=off}

is maximal in absolute value and which imaginary part is nonnegative. For the polynomials we apply this definition to, this uniquely determines the root.

\begin{exmp}
Consider the second row of Table~\ref{tab:1} and choose $\pm$ to be $+$, that is $\sigma(z) = \frac{z+1}{p/q \cdot z-1}$ and $p/q \neq - 1$. Then, $\mathcal{O} =\{\{0,-1\}\},\ \phi(z) = \frac{z(z+1)}{pz-q}$ and $D=\log(\max(|p- q|/2,|q|)$. Then
$$\log^+|z|+\log^+\left|\sigma(z)\right|-\log|\phi(z)|\geq D,$$
which yields $ h_G(0)=0$, $h_G(-1)=0$ and $h_G(\alpha)\geq D$ for all $\alpha \in \overline{\mathbb{Q}}$ with $\alpha \neq 0,-1$. We have $h_G(\alpha)= D$ for $\alpha=1$ if $2\mid p-q$. 
\end{exmp}
There are $\sigma \in \pglq$ which can be found twice in Table~\ref{tab:1}. For these $\sigma$ we have that the corresponding value of $D$ is not positive, so this value of $D$ is not allowed in Theorem~\ref{thm:1}. 
 %For example, the transformation
%$$\begin{pmatrix} 1 & 0 \\ 1 & -1 \end{pmatrix}$$
%corresponds to the first row for $p/q=1$ and to the seventh row for $p/q=0$. 
All these $\sigma$ can be found in Table~\ref{tab:2}. They all have order 2. Using another inequality, it is in some cases still possible to find the optimal (positive) value of $D$. Namely, it holds that
$\log^+|z|+\log^+|\sigma(z)|-B_1 \log|\phi_1(z)|-B_2 \log|\phi_2(z)|\geq D$ for corresponding values in the columns `$B_1$', `$B_2$' and `$\exp(D)$'. Here, $\phi_1$ and $\phi_2$ are similarly defined, that is, $\phi_i(z) = \frac{1}{E_i} \prod_{\sigma \in G}p_i(\sigma(z))$ where $p_i$ corresponds to the $i$th orbit in $\mathcal{O}$ as in the proof of Theorem~\ref{thm:1} and $E_i$ is such that the numerator and denominator of $\phi_i(z)$ are relatively prime. Again, it follows that
$ h_G(\alpha)=0$ or $h_G(\alpha)\geq D$ for all $\alpha \in \overline{\mathbb{Q}}$. In three of the cases, the author was not able to find the optimal value of $D$ with corresponding values of $B_1$ and $B_2$.

\section{Generalizations}
It is possible to extend Table~\ref{tab:1} and Table~\ref{tab:2} to non-cyclic subgroups of $\pglq$. For example consider
$$G= \left\langle\begin{pmatrix} 1 & -1\\ 3 & 1 \end{pmatrix}, \begin{pmatrix} 1 & 0 \\ 0 & -1\end{pmatrix}\right\rangle \simeq D_3,$$
the dihedral group with 6 elements. Then $\mathcal{O}=\{\{-1, 0, 1\}\}$ and one can show that $h_G(\alpha)\geq \log(25)$ for $\alpha \neq -1,0,1$ with equality for $\alpha=i$. Similarly, for
$$G=\left \langle\begin{pmatrix} 1 & -1 \\ 0 & -1\end{pmatrix}, \begin{pmatrix} 1 & 1\\ 2 & -1\end{pmatrix}\right \rangle \simeq D_2,$$
one has $\mathcal{O}=\{\{\omega_{+ +}, \omega_{+ -}\}\}$ and one can show that $h_G(\alpha)\geq \log(2)$ for $\alpha \neq \omega_{+ +}, \omega_{+ -} $ with equality for $\alpha=-1$.
It would be interesting to specify for which other subgroups one can find a similar statement. \\
Another way to extend these tables, is by considering finite $G\leq \pgloq$. For example, $\sigma(z)= \frac{z -\sqrt{3}}{\sqrt{3}z + 1}$ has order $3$ and one finds $\mathcal{O}=\{\{i\},\{-i\}\}$ for $G=\langle \sigma \rangle$. \\
Although the value of $D$ is computed in some cases, this note does not explain how $G$ determines the value of $D$. It would be interesting to find a universal lower bound on $D$ or to strengthen Theorem~\ref{thm:1} by proving that $D$ is greater than some invariant depending on $G$. 

\section{Acknowledgement}
I would like to thank my supervisor Gunther Cornelissen for helpful discussions and suggestions.

%\bibliographystyle{amsplain}
%\bibliography{bibpaper}

\def\cprime{$'$}
\providecommand{\bysame}{\leavevmode\hbox to3em{\hrulefill}\thinspace}
\providecommand{\MR}{\relax\ifhmode\unskip\space\fi MR }
\providecommand{\MRhref}[2]{%
  \href{http://www.ams.org/mathscinet-getitem?mr=#1}{#2}
}
\providecommand{\href}[2]{#2}

\end{document}